\documentclass{amsart}

\usepackage{amsmath, amssymb, amsthm}
\usepackage[latin1]{inputenc}
\usepackage{geometry}
\usepackage[normalem]{ulem}
\usepackage{cancel}
\usepackage[bookmarks=false,colorlinks=true,linkcolor=black,citecolor=black,filecolor=black,urlcolor=black]{hyperref}

\newtheorem{theorem}{Theorem}[section]
\newtheorem{lemma}[theorem]{Lemma}
\newtheorem{proposition}[theorem]{Proposition}

\newtheorem{corollary}[theorem]{Corollary}
\newtheorem{remark}[theorem]{Remark}

\newenvironment{proofof}{\par\noindent{\bf Proof of\;}}{\qed\par\bigskip}

\newcommand{\Cen}{\mbox{\rm C}}

\newcommand{\Aut}{\mbox{\rm Aut}}

\newcommand{\inv}{^{-1}}

\newcommand{\sgn}{{\rm sgn}}

\newcommand{\FC}{{\mathcal{F}}}
\newcommand{\qand}{\quad \text{and} \quad}
\DeclareMathOperator{\Supp}{\text{Supp}}

\newcommand{\matriz}[1]{\begin{array} #1 \end{array}}
\newcommand{\pmatriz}[1]{\left(\begin{array} #1 \end{array}\right)}
\newcommand{\GEN}[1]{\langle #1 \rangle}

\newcommand{\B}{{\mathcal B}}
\newcommand{\U}{\mathcal{U}}

\newcommand{\Z}{\mathcal{Z}}
\newcommand{\PP}{\mathcal{P}}

\date{}
\title[Group algebras whose units satisfy a Laurent Polynomial Identity]{Group algebras whose units satisfy a \\ Laurent Polynomial Identity}

\thanks{The first author has been partially supported by FAPEMIG(Proc. n$^0$ ETC 00163/15) and CAPES(Proc. n$^0$ BEX 4147/13-8) of Brazil.
The second author has been partially supported by Spanish Government under Grant MTM2016-77445-P with "Fondos FEDER" and, by Fundación Séneca of Murcia under Grant 19880/GERM/15.}

\author{Osnel Broche}
\address{Osnel Broche, Departamento de Ciências Exatas, Universidade Federal de Lavras,
Caixa Postal 3037, 37200-000, Lavras, Brazil }
\email{osnel@dex.ufla.br}

\author{Jairo Z. Gonçalves}
\address{Jairo Z. Gonçalves, Department of Mathematics, University of São Paulo, 05508-090, Brazil} \email{jz.goncalves@usp.br}

\author{Ángel del Río}
\address{Ángel del Río, Departamento de Matemáticas, Universidad de Murcia,
30100, Murcia, Spain}
\email{adelrio@um.es}



\begin{document}
\begin{abstract}
Let $KG$ be the group algebra of a torsion group $G$ over a field $K$. We show that if the units of $KG$ satisfy a Laurent polynomial identity which is not satisfied by the units of the relative free algebra $K[\alpha,\beta : \alpha^2=\beta^2=0]$ then $KG$ satisfies a polynomial identity. This extends Hartley Conjecture which states that if the units of $KG$ satisfies a group identity then $KG$ satisfies a polynomial identity. As an application of our results we prove that if the units of $KG$ satisfies a Laurent polynomial identity with a support of cardinality at most 3 then $KG$ satisfies a polynomial identity.
\end{abstract}

\maketitle

\section{Introduction}

Let $K$ be a field.
Let $A$ be a $K$-algebra and let $\U(A)$ denote the group of units of $A$.
Let $K\GEN{X_1,X_2,\dots}$ denote the $K$-algebra of polynomials with coefficients in $K$ in countably many non-commuting variables.
If $f\in K\GEN{X_1,X_2,\dots}$ and $a_1,a_2,\dots \in A$ then the result of evaluating $f$ with $X_i=a_i$ for every $i$, is denoted $f(a_1,a_2,\dots)$.
A polynomial identity (PI, for short) of $A$ is an element of $K\GEN{X_1,X_2,\dots}$ such that  $f(a_1,a_2,\dots)=0$ for every $a_1,a_2,\dots\in A$.
The algebra $A$ is PI if some non-zero polynomial $f\in K\GEN{X_1,X_2,\dots}$ is a PI for $A$.

Let $K\GEN{X_1^{\pm 1},X_2^{\pm 1},\dots}$ be the $K$-algebra of Laurent polynomials with coefficients in $K$ in countably many non-commuting variables.
If $f\in K\GEN{X_1^{\pm 1},X_2^{\pm 1},\dots}$ and $u_1,u_2,\dots \in \U(A)$ then the evaluation $f(u_1,u_2,,\dots)$ of $f$ with $X_i=u_i$ makes sense.
A \emph{Laurent polynomial identity} (LPI, for short) of $\U(A)$ is an element $f\in K\GEN{X_1^{\pm 1},X_2^{\pm 1},\dots}$ such that $f(u_1,u_2,\dots)=0$ for every $u_1,u_2,\dots\in \U(A)$.
We say that $\U(A)$ \emph{satisfies a Laurent polynomial identity}  (LPI) if some non-zero Laurent polynomial is a LPI of $\U(A)$.

If $G$ is a group then $KG$ denotes the group algebra of $G$ with coefficients in $K$.
This notation is compatible with the one used for the algebra of Laurent polynomials because $\GEN{X_1^{\pm 1},X_2^{\pm 1},\dots}$ is a free group in countably many variables and the algebra $K\GEN{X_1^{\pm 1},X_2^{\pm 1},\dots}$ of Laurent polynomials coincides with the group algebra of $\GEN{X_1^{\pm 1},X_2^{\pm 1},\dots}$ with coefficients in $K$.
As every free group of rank $n$ is contained in the free group of rank $2$, $\U(A)$ satisfies a LPI if and only if $\U(A)$ satisfies a LPI in $2$ variables.

Group algebras satisfying a polynomial identity where characterized by Isaacs and Passman \cite{IsaacsPassman1964,Passman1972}.
Brian Hartley conjectured in the 1980's that if $G$ is periodic and $\U(KG)$ satisfies a group identity then $KG$ satisfies a polynomial identity.
This conjecture was studied by many authors who obtained several partial results (see e.g.
\cite{Warhurst1981,GoncalvesMandel1991,GiambrunoJespersValenti1994}).
Giambruno, Sehgal and Valenti proved Hartley Conjecture under the assumption that $K$ is infinite \cite{GiambrunoSehgalValenti1997}.
Hartley Conjecture was finally proved by Liu in \cite{Liu1999}.
The converse of Hartley Conjecture is not true. For example, if $G$ is finite then $KG$ satisfies obviously a polynomial identity but in most cases $\U(KG)$
contains a free group. For example, this is the case if $K$ has zero characteristic and $G$ is neither abelian nor a Hamiltonian $2$-group
\cite{MarciniakSehgal1997}.

Recall that if $a=\sum_{g\in G} a_g g$ is an element of a group algebra $KG$ with $a_g\in K$ for each $g\in G$ then the \emph{support} of $a$ is $\Supp(a)=\{g\in G : a_g\ne 0\}$. In particular, the support of a Laurent polynomial $f$ is the set of monomials with non-zero coefficient in $f$.
Observe that a group identity is a special case of a LPI, namely a LPI with two elements in the support.
This suggest to ask whether it is sufficient that $\U(KG)$ satisfies a LPI for $KG$ to satisfy a PI.
We prove that this is correct provided that the LPI is not an identity for the units of the following algebra:
	$$\FC = K[\alpha,\beta : \alpha^2=\beta^2=0].$$
The algebra $\FC$ is the free algebra in two non-commutative variables $\alpha$ and $\beta$ relative to the relations $\alpha^2=\beta^2=0$.
Formally we prove the following

\begin{theorem}\label{thetheorem}
Let $K$ be a field, let $G$ be a torsion group, let $KG$ be the group algebra of $G$ with coefficients in $K$
and let $\FC=K[\alpha,\beta:\alpha^2=\beta^2=0]$.
If $\U(KG)$ satisfies a LPI which is not a LPI of $\U(\FC)$ then $KG$ satisfies a PI.
\end{theorem}

Theorem~\ref{thetheorem} generalizes Hartley Conjecture (i.e. the main result of \cite{Liu1999}) because $\U(\FC)$ contains a free group
and consequently $\U(\FC)$ does not satisfy any non-trivial group identity
(see \cite[Theorem~1.1]{GoncalvesdelRio2011} and  \cite{GoncalvesPassman1996}).
Actually using Theorem~\ref{thetheorem} and some properties of the LPIs of $\FC$ we will prove the following theorem:

\begin{theorem}\label{LPI3Support}
If $K$ is a field and $G$ is a torsion group such that $\U(KG)$ satisfies a LPI whose support has at most three elements then $KG$ satisfies a PI.
\end{theorem}

Some arguments for the proof of Theorem~\ref{thetheorem} uses ideas from \cite{Liu1999}. Essential for adapting these ideas are the following two conditions:
Let $g$ be a polinomial in one variable with coefficients in $K$.
We say that a $K$-algebra $A$ \emph{has the property  $\PP_1$ with respect to} $g$  if $g(ab)=0$ for every $a,b\in A$, with $a^2=0=b^2$. We say that a Laurent polynomial $f$ over $K$ \emph{has the property ${\PP}$} if every $K$-algebra $A$, for which $f$ is a LPI of $\U(A)$, has the property $\PP_1$ with respect to some non-zero polynomial.

In Section~\ref{SectionTorsionFree} we prove the following two key results:

\begin{lemma}\label{Restriction}
	Let $f$ be a Laurent polynomial over $K$.
	Then the following are equivalent:
	\begin{enumerate}
		\item $f$ is not a LPI of $\U(\FC)$.
		\item There is  a non-zero polynomial $g\in K[T]$ such that every $K$-algebra $B$ for which $f$ is a LPI of $\U(B)$ has the property $\PP_1$ with respect to
		$g$.
	\end{enumerate}
\end{lemma}

\begin{proposition}\label{Gp'grouplocfin}
Let $K$ be a field of characteristic $p$ and let $G$ be a locally finite $p'$-group.
If $KG$ has the property $\PP_1$ with respect to some non-zero polynomial then $KG$ satisfies a standard polynomial identity.
\end{proposition}

This yields at once Theorem~\ref{thetheorem} provided that $G$ is a locally finite $p'$-group.
	
In Section~\ref{SectionProofOfTheTheorem} we prove the following:

\begin{theorem}\label{soe}
Let $KG$ be the group algebra of the torsion group $G$ over the field $K$.
If $KG$ satisfies a LPI which has the property $\PP$ then $KG$ satisfies a PI.
\end{theorem}

Then the proof of Theorem~\ref{thetheorem}  follows easily.
Indeed, suppose that $\U(KG)$ satisfies a LPI $f$ which is not an identity of $\U(\FC)$.
Then $f$ has the property $\PP$, by Lemma~\ref{Restriction}. Hence $KG$ satisfies a PI, by Theorem~\ref{soe}.

In order to use Theorem~\ref{thetheorem} at a practical level it would be convenient to have a handy criteria to decide whether a Laurent polynomial is an LPI for $\U(\FC)$.
It is clear that the criteria given by Lemma~\ref{Restriction} cannot be considered handy.

In Section~\ref{SectionExamples} we provide some necessary conditions for a Laurent polynomial to be an LPI for $\U(\FC)$, whose verification for a particular Laurent polynomial is straightforward (Proposition~\ref{ThExamples}).
As an application we prove the following:
\begin{proposition}\label{Support4}
Every non-zero LPI of $\U(\FC)$ has a support with at least four elements.
\end{proposition}
Theorem~\ref{LPI3Support} follows directly from Theorem~\ref{thetheorem} and Proposition~\ref{Support4}.

\section{The locally finite $p'$-case}\label{SectionTorsionFree}

In this section we prove Lemma~\ref{Restriction} and Proposition~\ref{Gp'grouplocfin}. As explained in the introduction this proves Theorem~\ref{thetheorem} under the additional assumption that $G$ is a locally finite $p'$-group, where $p$ is the characteristic of $K$, and in particular, in case $K$ has characteristic 0 and $G$ is locally finite.

For the proof of Lemma~\ref{Restriction} we need the following:

\begin{lemma}\label{thekey}
If $L=\{r\in \FC : \alpha(1+\beta)r(1+\alpha)\beta=0\}$ and $\Aut_K(\FC)$ denotes the group of $K$-automorphisms of $\FC$ then
  $$\bigcap_{\gamma\in \Aut_K(\FC)} \gamma(L)=0.$$
\end{lemma}

\begin{proof}
Consider the homomorphism $\varphi:\FC\mapsto M_2(K[T])$ with
  $$a=\varphi(\alpha)=\pmatriz{{cc} 0 & 1\\ 0 & 0} \quad \text{and} \quad b=\varphi(\beta)=\pmatriz{{cc} 0 & 0 \\ T & 0}.$$
It is easy to see that $\varphi$ is injective and
  $$\varphi(\FC)=\left\{\pmatriz{{cc} x+T A & B \\ T C & x+TD} :  x\in K; A,B,C,D\in K[T]\right\}.$$
Fix a generic element $s=\pmatriz{{cc} x+T A & B \\ T C & x+TD}$ of $\varphi(\FC)$.
Then
  $$a(1+b)s(1+a)b %= \pmatriz{{cc} T & 1 \\ 0 & 0} \pmatriz{{cc} x+T A & B \\ T C & x+TD}\pmatriz{{cc} T & 0 \\ T & 0}
  =T\pmatriz{{cc} x+T(x+T A+B+C+D) &  0 \\0 & 0}.$$
Therefore $s\in \varphi(L)$ if and only if $x=0$ and $T A+B+C+D=0$.

Suppose that $r\in \bigcap_{\gamma\in \Aut_K(\FC)} \gamma(L)$ and that $s=\varphi(r)$ is as above. Then all the conjugates of $s$ in $\varphi(\FC)$ belong to $\varphi(L)$.
In particular,
$$(1+a)s(1-a)=\pmatriz{{cc} T(A+C) & B+T(D-A-C) \\ T C & T(D-C)}\in \varphi(L)$$
and
$$(1-a)s(1+a)=\pmatriz{{cc} T(A-C) & B+T(A-D-C) \\ T C & T(D+C)}\in \varphi(L).$$
Therefore
\begin{eqnarray*}
T A+B+C+D & = & 0 \\
B+(T+1)D &=& 0 \\
2TA+B+2(1-T)C+(1-T)D&=&0
\end{eqnarray*}
Solving this equation we obtain $C=0$, $A=D$ and $B=-(1+T)A$.
Therefore
	$$s=\pmatriz{{cc} TA & -(1+T)A \\ 0 & TA}.$$
Conjugating with $1+b$ we obtain
$$(1+b)s(1-b)
% =\pmatriz{{cc} 1 & 0 \\ T & 1} \pmatriz{{cc} TA & -(1+T)A \\ 0 & TA} \pmatriz{{cc} 1 & 0 \\ -T & 1}
% \pmatriz{{cc} TA & -(1+T)A \\ T^2A & -T^2A} \pmatriz{{cc} 1 & 0 \\ -T & 1}
=\pmatriz{{cc} (T^2+2T)A & -(1+T)A \\ (T^2+T^3)A &-T^2A}\in \varphi(L)$$
Therefore  $(2T^2+T-1)A=0$ so that $A=0$ and $s=0$. Since $\varphi$ is injective, $r=0$, as desired.
\end{proof}

%\begin{lemma}\label{Restriction}
%Let $f$ be a Laurent polynomial over $K$.
%Then the following are equivalent:
%\begin{enumerate}
% \item $f$ is not a LPI of $\U(\FC)$.
% \item There is  a non-zero polynomial $g\in K[T]$ such that every $K$-algebra $B$ for which $f$ is a LPI of $\U(B)$ has the property $\PP_1$ with respect to
%$g$.
%\end{enumerate}
%\end{lemma}

\begin{proofof}\textbf{Lemma~\ref{Restriction}}.
(2) implies (1)
Suppose that $f$ is a LPI of $\U(\FC)$. As $\alpha\beta$ is transcendental over $K$, $\FC$ does not have the property
$\PP_1$ with respect to any non-zero polynomial. Hence $f$ does not satisfy condition (2).

(1) implies (2)
Suppose that $f$ is not a LPI of $\U(\FC)$ and that the number of variables of $f$ is $n$.
Let $u_1,\dots,u_n\in \U(\FC)$ such that $f(u_1,\dots,u_n)\ne 0$.
Let $L$ be as in Lemma~\ref{thekey}.
Then $f(\gamma(u_1),\dots,\gamma(u_n))=\gamma(f(u_1,\dots,u_n))\not\in L$ for some $K$-automorphism $\gamma$ of $\FC$ and therefore we
may assume without loss of generality that $r=f(u_1,\dots,u_n)\not\in L$. Therefore $\alpha(1+\beta)r(1+\alpha)\beta\ne 0$.
Thus there are $\sigma \in \{\alpha,\alpha\beta\}$ and $\tau \in \{\beta,\alpha\beta\}$ such that $\sigma r\tau\ne 0$.
Observe that $\sigma r \tau=g(\alpha\beta)$ for some $g\in K[T]$ and necessarily $g\ne 0$.

Let $B$ be a $K$-algebra such that $f$ is a LPI of $\U(B)$ and let $a,b\in B$ with $a^2=b^2=0$. Then there is a homomorphism of
$K$-algebras $\varphi:\FC\rightarrow B$ such that $\varphi(\alpha)=a$ and $\varphi(\beta)=b$.
As $\varphi(u_1),\dots,\varphi(u_n)$ are units of $B$, we have $g(ab)=g(\varphi(\alpha)\varphi(\beta))=g(\varphi(\alpha\beta))=\varphi(g(\alpha\beta))=\varphi(\sigma r
\tau)=\varphi(\sigma)\varphi(f(u_1,\dots,u_n))\varphi(\tau)=\varphi(\sigma)f(\varphi(u_1),\dots,\varphi(u_n))\varphi(\tau)=0$.
This proves the Lemma.
\end{proofof}

For the proof of Proposition~\ref{Gp'grouplocfin} we need the following:

\begin{lemma}\label{finitecondi}
If $A$ is a $K$-algebra containing $M_n(K)$ with $n\geq 2$ and $A$ has the property $\PP_1$ with respect to the polynomial $g\in K[T]\setminus \{0\}$ then
\begin{enumerate}
\item\label{finiteK} $|K|\leq deg(g)$ and
\item\label{matrixorder} $n < 2\log_{|K|}(deg(g))+2$.
\end{enumerate}
\end{lemma}

\begin{proof}
Let $e_{ij}$ denote the matrix having $1$ at the $(i,j)$-th entry and zeros at all other entries.
Let $a=r e_{12}$ and $b=e_{21}$, with $r\in K$. Then $a^2=0=b^2$ and $ab=r e_{11}$.
Thus
$0=g(ab)=g(r)e_{11}$ and therefore $r$ is a root of $g$ for all $r\in K$. Hence $|K|\leq deg(g)$.

Now, let $s=[\log_{|K|}(deg(g))] + 1$ and let $E$ be an extension of $K$ of degree $s$.
Then $|K|^{s-1} \leq deg(g) <|K|^s=|E|$.
Moreover $M_2(E)\hookrightarrow M_2(M_s(K))=M_{2s}(K)$, by the regular representation of  $E$ over $K$.
Thus, if $n\geq 2s$ then $M_2(E)\subseteq A$ and, by (\ref{finiteK}),
$|E|\leq deg(g)$, a contradiction.
Hence $n < 2s \le 2\log_{|K|}(deg(g))+2$.
\end{proof}

%\begin{lemma}\label{}
%Suppose that $A$ has the property $\PP_1$ with respect to $g\in K[T]\setminus \{0\}$. Let $a,b\in A$ such that $a^2=0=b^2$ and $ab$ is nilpotent.
%Then $(ab)^{deg(g)}=0$.
%\end{lemma}
%
%\begin{proof}
%Let $h=\Min_K(ab)$, the minimal polynomial of $ab$ over $K$, and let $k$ be a integer such that $(ab)^k=0$.
%Thus $h$ divides both $T^k$ and $g$, since $g(ab)=0$.
%Hence $h=T^m$ for some $m\le \min\{k,\deg(g)\}$. Then $(ab)^{deg(g)}=0$.
%\end{proof}

%\begin{proposition}\label{Gp'grouplocfin}
%Let $K$ be a field of characteristic $p$ and let $G$ be a locally finite $p'$-group.
%If $KG$ has the property $\PP_1$ with respect to some polynomial then $KG$ satisfies a standard polynomial identity.
%\end{proposition}

\begin{proofof}\textbf{Proposition~\ref{Gp'grouplocfin}}.
Assume that $KG$ has the property $\PP_1$ with respect to $g\in K[T]\setminus \{0\}$ and let $m=[\log_2(deg(g))]+1$. We will show that $KG$ satisfies the standard
polynomial identity $S_{4m}$ of degree $4m$.
Let $\alpha_1,\cdots,\alpha_{4m}\in KG$, and let $H$ be the subgroup of $G$
generated by all the elements in the union of the supports of the $\alpha_i$'s.
Since $G$ is a locally finite $p'-$group, $H$ is a finite $p'$-group.
By the Maschke Theorem and the Wedderburn-Artin Theorem,
$KH\cong\prod M_{n_i}(F_i)$, where each $F_i$ is a finite field.
By Lemma~\ref{finitecondi}.(\ref{matrixorder}), $n_i< 2\log_{|F_i|}(deg(g))+2\leq 2m$ and  hence, by the Amitsur-Levitzki Theorem, we have  $S_{4m}(\alpha_1,\cdots,\alpha_{4m})=0$, as desired.
\end{proofof}

%
%\noindent \textbf{Property $\PP$}:
%Let $P$ be a Laurent polynomial over a field $K$ we say that $P$ has the property ${\PP}$ if every $K$-algebra $A$, for which $P$ is a \ChOsnel { LPI of $\U(A)$\cancel{Laurent identity}}, \ChOsnel{\sout{satisfies} has the property} $\PP_1$. \ChOsnel{Mejor asi?}
%\medskip

\section{Proof of Theorem~\ref{soe}}\label{SectionProofOfTheTheorem}

In this section we prove Theorem~\ref{soe}. This completes the proof of Theorem~\ref{thetheorem}, as explained in the introduction.

All throughout the section $G$ is a torsion group and, as in the previous section, $K$ is a field and $p$ denotes the characteristic of $K$.

We start with a very elementary lemma.

\begin{lemma}\label{(bac)PI}
	Suppose that $A$ has the property $\PP_1$ with respect to $g\in K[T]\setminus \{0\}$ and let $h=Tg(T)$.  Then $h(bacA)=0$ for all  $a,b,c\in A$, such that
	$a^2=0=bc$.
\end{lemma}

\begin{proof}
	For any $r\in A$ we have that $(crb)^2=0$. Thus, $g(acrb)=0$ and hence  $h(bacr)=bacrg(bacr)=bg(acrb)acr=0$.
\end{proof}

We use $\Cen_X(Y)$ to denote the centralizer of $Y$ in $X$, whenever $X$ and $Y$ are subsets of a group or a ring. In case $X$ is a group and $Y\subseteq X$ then $N_X(Y)$ denotes the normalizer of $Y$ in $X$.

If $g\in G$ then $\widehat{g}$ denotes the sum of the elements of $\GEN{g}$ in $KG$.
It is easy to see that if $g,h\in G$ then $h\in \Cen_G(\widehat{g})$  if and only if $h\in N_G(\GEN{g})$ if and only if $(g-1)h\widehat{g}=0$.

Let $N(KG)$ denote the nilpotent radical of $KG$, namely, the sum of all nilpotent ideals of $KG$.
Let also
$$\Delta(G)=\{g\in G : g \text{ has only finitely many conjugates in } G\} = \{g\in G\ |\ [G:\Cen_G(g)]<\infty \},$$
the f.c. subgroup of $G$, and
$$\Delta^p(G)=\GEN{g\in \Delta(G) : g \text{ is a }p\text{-element}},$$
with $\Delta^0(G)=1$.
Recall that $G$ is called an f.c. group, if all the conjugacy classes of $G$ are finite, i.e. if $G=\Delta(G)$.

\begin{lemma}\label{idempotentnilpotent}
Suppose that $KG$ satisfies the following conditions:
\begin{enumerate}
\item\label{centralidempotent} Every idempotent of $KG$ is central;
\item\label{squarezeroandnilpotent}
If $a,b,c\in KG$ with $bc=0$ and $a$ is nilpotent then $bac=0$.
\end{enumerate}
Then $G$ is an f.c. group.
\end{lemma}

\begin{proof}
Let $P$ be the set of $p$-elements of $G$ and let $Q$ be the set of $p'$-elements of $G$.
Clearly  $P^{-1}=P^{g}=P$ and $Q^{-1}=Q^{g}=Q$ for all $g\in G$.
Moreover every element of $G$ is of the form $pq=qp$ for $p\in P$ and $q\in Q$.
If $x \in Q$ has order $m$ then $\frac{1}{m}\widehat{x}$ is an idempotent of $KG$ and, by (\ref{centralidempotent}), it is central in $KG$.
Thus $\GEN{x}$ is normal in $G$.
This shows that every subgroup of $G$ contained in $Q$ is normal in $G$ and using this it is easy to see that $Q$ is closed under products and hence $Q$ is a normal subgroup of $G$.
Suppose now that $x\in P$ and $g\in G$.
Then $x-1$ is nilpotent and $(g-1)\widehat{g}=0$.
Thus, by (\ref{squarezeroandnilpotent}), we have $0=(g-1)(x-1)\widehat{g}=(g-1)x\widehat{g}$. Hence $x$ normalizes $\GEN{g}$.
In particular, $\GEN{x}$ is normal in $P$. As above, this implies that all the subgroups of $P$ are normal in $P$ and $P$ is closed under products. Hence $P$ and $Q$ are normal subgroups of $G$ with $G=PQ$ and $P\cap Q=1$.
Thus $G=P\times Q$ and every subgroup of $P$ (respectively, $Q$) is normal in $G$. Then $G$ is an f.c. group, as $G$ is torsion.
\end{proof}

\begin{lemma}\label{IdealPI}
If $\U(KG)$ satisfies a LPI which has the property $\PP$ and $KG$ contains a right ideal $I$ satisfying a PI of degree $k$ such
that $I^k\neq
0$ then $KG$ satisfies a PI.
\end{lemma}

\begin{proof}
Recall that $p$ denotes the characteristic of $K$.
Let $\Delta=\Delta(G)$, $C=\Cen_\Delta(\Delta')$ and let $P$ be the set of $p$-elements of $C$.
In case $p=0$ then $P=1$.
Since $C'\subseteq \Z(C)$, $C$ is nilpotent and hence $P$ is a subgroup of $C$.
By \cite[Lemma~2.8]{Liu1999}, the hypothesis on the ideal $I$ implies that $KG$ satisfies a generalized polynomial identity.
Thus $[G:\Delta]<\infty$ and $|\Delta'|<\infty$, by  \cite[Theorem~5.3.15]{Passman77}.
The former implies that $G$ is locally finite, as so is $\Delta$ by \cite[Lemma~4.1.8]{Passman77}.
The latter yields $[G:C]=[G:\cap_{g\in \Delta'} \Cen_G(g)]\le \prod_{g\in \Delta'} [G:\Cen_{G}(g)]<\infty$ and $G$ is locally finite.
As $C$ is nilpotent, the natural map $\U(KC)\rightarrow\U(K(C/P))$ is surjective.
Thus $\U(K(C/P))$ satisfies a LPI which has the property $\PP$.
Therefore $K(C/P)$ has the property $\PP_1$ with respect to some polynomial and $C/P$ is a locally finite $p'$-group.
Then, by Proposition~\ref{Gp'grouplocfin}, we have that $K(C/P)$ satisfies a PI.
Hence, $C/P$ has a subgroup of finite index $A/P$  with $(A/P)'$ a finite $p$-group, by \cite[Corollary~5.3.10]{Passman77}.
(In case $p=0$, $A$ is abelian by \cite[Corollary~5.3.8]{Passman77}.)
Thus $[G:A]=[G:C][C:A]<\infty$ and $A/P$ is abelian.
Hence $A'\subseteq P\cap C'\subseteq P\cap \Delta'$.
Therefore $A'$ is a finite $p$-group.
Therefore $G$ has a $p$-abelian subgroup of finite index.
Hence $KG$ satisfies a PI \cite[Corollary~5.3.10]{Passman77}.
\end{proof}

%\begin{theorem}\label{soe}
%Let $KG$ be the group algebra of the torsion group $G$ over the field $K$.
%If $KG$ satisfies a LPI which has the property $\PP$ then $KG$ satisfies a PI.
%\end{theorem}

\begin{proofof} \textbf{Theorem~\ref{soe}}.
By Lemma~\ref{IdealPI} we may assume that for every right ideal $I$ of $KG$, if $I$ satisfies a PI of degree $k$ then $I^k=0$.
Let $f$ be a LPI of $KG$ which has the property $\PP$.

We first prove the theorem under the assumption that $\Delta^p(G)=1$.
Then $KG$ is semiprime and hence $\Delta(G)$ is a $p'$-group. This is clear if $p=0$ and a consequence of \cite[Theorem~4.2.13]{Passman77} if $p>0$.
On the other hand, by Lemma~\ref{(bac)PI} for all $a,b,c\in KG$, such that $a^2=0=bc$, we have that the ideal $bacKG$ satisfies a PI in one variable.
Thus $bacKG$ is a nilpotent ideal, by the hypothesis in the first paragraph.
As $KG$ is semiprime, we deduce that $bac=0$.
By \cite[Lemma~2]{GiambrunoJespersValenti1994}, every idempotent of $KG$ is central and, by
\cite[Lemma~2.1]{GiambrunoSehgalValenti1997}, for all $a,b,c\in KG$, with $a$ nilpotent and $bc=0$, we have that $bac=0$.
Thus, by Lemma~\ref{idempotentnilpotent} we have that $G$ is an f.c. torsion group and, in particular, it is locally finite \cite[Lemma~4.1.8]{Passman77}.
Therefore $G=\Delta(G)$, a $p'$-group.
Hence, $KG$ satisfies a PI, by Proposition~\ref{Gp'grouplocfin}.

We now suppose that $\Delta^p(G)$ is finite but non-trivial. The latter implies that $p>0$.
Then the kernel of the natural map $KG\rightarrow K(G/\Delta^p(G))$ is nilpotent by \cite[Lemma~3.1.6]{Passman77}.
Thus  $\U(KG)\rightarrow\U(K(G/\Delta^p(G)))$ is surjective and $\U(K(G/\Delta^p(G)))$  satisfies $f$.
Moreover $\Delta^p(G/\Delta^p(G))=1$.
Therefore,  $K(G/\Delta^p(G))$ satisfies a PI, by the previous paragraph. Then, by \cite[Corollary~5.3.10]{Passman77},
 $G/\Delta^p(G)$ has a $p$-abelian subgroup $A/\Delta^p(G)$ of finite index.
Therefore $A'\Delta^p(G)/\Delta^p(G)$ is a finite $p$-group. Thus $A'$ is a finite $p$-group and hence $KG$ satisfies a PI, again by
\cite[Corollary~5.3.10]{Passman77}.

In the remainder of the proof we argue by contradiction to prove that $\Delta^p(G)$ is finite, and so the theorem follows by the previous paragraphs.
For that we use the following ring homomorphism, where $n$ is the number of variables of $f$ and $T_1,\dots,T_n$ are independent commuting variables
\begin{eqnarray*}
    K\GEN{X_1^{\pm 1},\dots,X_n^{\pm 1}}&\stackrel{\Phi}{\longrightarrow} &K\GEN{X_1,\cdots, X_n}[[T_1,\dots,T_n]] \\
    X_i& \mapsto & 1+X_iT_i
    \end{eqnarray*}
By the argument of \cite[Lemma~1]{Makar-Limanov1988}, $\Phi$ is injective.
Moreover,
$$\Phi(f)=\sum p_{i_1i_2\cdots i_n}(X_1,\cdots, X_n)T^{i_1}T^{i_2}\cdots T^{i_n},$$
where each $p_{i_1i_2\cdots i_n}(X_1,\cdots, X_n)$ is a homogeneous polynomial in $K\GEN{X_1,\cdots, X_n}$ which has degree $i_j$ in the variable $X_j$ for every $j=1,\dots,n$.
Take fixed indices $m_1,m_2,\cdots, m_n$ such that $p_{m_1m_2\cdots m_n}$ is a non-zero polynomial with $m=m_1+m_2+\dots+m_n$ minimal. Then $m\ne 0$
because $p_{00\cdots0}(X_1,\cdots, X_n)=p_{00\cdots0}(1,\cdots,1)=\Phi(f(1,\cdots, 1))(0,\cdots,0)=0$.

Recall that we are assuming that $\Delta^p(G)$ is not finite.
Then $N(KG)$ is not nilpotent by \cite[Theorem~8.1.12]{Passman77}.
Therefore, by the argument of \cite[Fact 3]{YehChuang1996},
there exists a nilpotent ideal $I$ of $KG$ such that
$I^{m}\neq 0$ and $I^{m+1} = 0$.
Thus, for any $a_1,\cdots, a_n\in I$ and
for any  indices $i_1,i_2,\cdots, i_n$ such that $i_1+\dots+i_n>m$,
we have that $p_{i_1i_2\cdots i_n}(a_1,\cdots, a_n)=0$.
Therefore $f(a_1,\dots,a_n)(T_1,\dots,T_n)\in A[T_1,\dots,T_n]$ and $0=f(a_1,\cdots, a_n)=\Phi(f(a_1,\cdots,
a_n))(1,1,\dots,1)=\sum_{i_1+\dots+i_n=m}  p_{i_1i_2\cdots i_n}(a_1,\cdots, a_n).$
Hence, $\sum_{i_1+\dots+i_n=m} p_{i_1i_2\cdots i_n}(X_1,\cdots, X_n)$
is a PI of $I$, of degree $m$, a contradiction with the hypothesis in the first paragraph of the proof.
This finishes the proof of the theorem.
\end{proofof}

\section{On LPIs of $\U(\FC)$}\label{SectionExamples}

In this section we prove Proposition~\ref{Support4} by first giving a list of necessary conditions for a Laurent polynomial in two variables to be an LPI of $\U(\FC)$. In order to state this we need to introduce some notation.

As in the previous sections $K$ is a field.
Moreover, $F=\GEN{X^{\pm 1},Y^{\pm 1}}$ is a free group of rank 2.

Let $w\in F$. Then $w$ has a unique expression as follows:
    \begin{equation}\label{Word}
	w=X^{n_0}Y^{m_1}X^{n_1}\cdots Y^{m_{k-1}}X^{n_{k-1}}Y^{m_k}X^{n_k}
\end{equation}
where $k\ge 0$, $n_0$ and $n_k$ are integers and $m_1,n_1,\dots,m_{k-1},n_{k-1},m_k$ are non-zero integers.
We call this expression the \emph{normal form} of $w$.
A \emph{subword} of $w$ is an element of one of the following forms:
$$X^{n_{i-1}}Y^{m_i}X^{n_i} \dots Y^{m_j}, \quad   X^{n_{i-1}}Y^{m_i}X^{n_i} \dots Y^{m_j}X^{n_j}, \quad
Y^{m_i}X^{n_i} \dots Y^{m_j}, \quad Y^{m_i}X^{n_i} \dots  Y^{m_j}X^{n_j}$$
with $1\le i \le j \le k$.

If $w\ne 1$ then the \emph{beginning} and \emph{end} of $w$ are respectively
$$B(w)=\begin{cases} X, & \text{if } n_0>0; \\ X^{-1}, & \text{if } n_0<0; \\
Y, & \text{if } n_0=0 \text{ and } m_1>0; \\ Y^{-1}, & \text{if } n_0=0 \text{ and } m_1<0;
\end{cases}
\qand
E(w)=\begin{cases} X, & \text{if } n_k>0; \\ X^{-1}, & \text{if } n_k<0; \\
Y, & \text{if } n_k=0 \text{ and } m_k>0; \\ Y^{-1}, & \text{if } n_k=0 \text{ and } m_k<0.
\end{cases}
$$
We also set $B(1)=E(1)=1$.

We associate to $w$ the following integers:
$$\sgn(w) = (-1)^{N(w)} \qand
C(w) =  \sum_{i=0}^k |n_i| + \sum_{i=1}^k |m_i| - M,$$
where $N(w)$ is the number of subwords of $w$ of form $X^nY^m$ or $Y^mX^n$ with $n>0$  and $m<0$; and $M$ is the number of subwords of $w$ of form $X^nY^m$  or  $Y^nX^m$ with $n<0$ and $m>0$.
We call $C(w)$ the \emph{cumulus} of $w$.

The following lemma collects some elementary properties of these concepts.

\begin{lemma}\label{CumulusSignPropiedades}
	Let $w\in F$ and let $c=C(w)$.
	\begin{enumerate}
		\item\label{Cumulus0} $c=0$ if and only if $w=1$.
		\item\label{Cumulus1} $c=1$ if and if  $w\in \{X, X^{-1}, Y, Y^{-1}, X^{-1}Y,Y^{-1}X\}$.
		\item\label{CumulusInverse} $C(w)=C(w^{-1})$.
		\item\label{CumulusDescomposicion}
		$w$ has a unique expression as product of $c$ elements of cumulus $1$ and cannot be written as product of less than $c$ elements of cumulus $1$.
		In particular, $w=w_1w'$ for unique $w_1,w'\in F$ with $C(w_1)=1$ and $C(w')=c-1$.
		\item\label{SignoInduccion}
		If $w_1$ and $w'$ are as in \eqref{CumulusDescomposicion} then $B(w)=B(w_1)$, $E(w)=E(w')$ and 		
		$$\sgn(w)=\begin{cases}
		-\sgn(w_1)\;\sgn(w'), & \text{if }
		E(w_1)=X \text{ and } B(w')=Y\inv;	\\
		\sgn(w_1)\;\sgn(w'), & \text{otherwise.}
		\end{cases}$$
	\end{enumerate}
\end{lemma}

\begin{proof}
\eqref{Cumulus0}, \eqref{Cumulus1} and \eqref{CumulusInverse} are straightforward.

\eqref{CumulusDescomposicion}
We argue by induction on $c$ with the case $c=1$ being trivial.
Suppose $c>1$ and that the result is true for elements of cumulus less than $c$.
Assume also that the normal form of $w$ is as in \ref{Word}.
Let
$$w_1=\begin{cases}
X, & \text{if } B(w)=X; \\
Y, & \text{if } B(w)=Y; \\
Y^{-1}, & \text{if } B(w)=Y^{-1} \text{ and } B(Yw) \in \{1,Y^{-1},X^{-1}\};\\
Y^{-1}X, & \text{if } B(w)=Y^{-1} \text{ and } B(Yw) =X;\\
X^{-1}, & \text{if } B(w)=X^{-1} \text{ and } B(Xw) \in \{1,X^{-1},Y^{-1}\}; \\
X^{-1}Y, & \text{if } B(w)=X^{-1} \text{ and } B(Xw) =Y;
\end{cases}$$
and $w'=w_1^{-1}w$.
Observe that $w_1$ is the unique element of cumulus $1$ satisfying the following conditions:
\begin{itemize}
	\item[(a)] $B(w)=B(w_1)$.
	\item[(b)] if $B(w)=X^{-1}$ then $B(Xw)=Y$ if and only if $B(Xw_1)=Y$.
	\item[(c)] if $B(w)=Y^{-1}$ then $B(Yw)=X$ if and only if $B(Yw_1)=X$).
\end{itemize}
Moreover, $C(w_1)=1$ and a case-by-case argument shows that $C(w')=c-1$. By the induction hypothesis $w'=w_2\dots w_c$ for unique elements $w_2,\dots,w_c$
of cumulus $1$. Thus $w$ is a product of $c$ elements of cumulus $1$.
Let $d$ be minimal so that $w$ is a product of $d$ elements of cumulus $1$.
Suppose that $w=v_1v_2\dots v_d$ with $C(v_i)=1$ for each $i$.
This implies that $v_iv_{i+1}\ne 1$ and $C(v_iv_{i+1})\ge 2$ for every $i<d$.
The first implies $B(w)=B(v_1)$. Using the second, we have that if $B(w)=X^{-1}$ then either $v_1=X^{-1}$ and
$B(v_2) \in \{X^{-1}, Y^{-1}\}$ or $v_1=X^{-1}Y$ and
$v_2 \in \{X, Y, X^{-1},X^{-1}Y\}$.
In both cases $B(Xw)=Y$ if and only if $B(Xv_1)=Y$. Similarly, if $B(w)=Y^{-1}$ then $B(Yw)=X$ if and only if $B(Yv_1)=X$.
This shows that $v_1$ satisfies the conditions stated in (a) and (b) for $w_1$, so that $v_1=w_1$ and hence $w'=v_2\dots v_d$.
By induction hypothesis $c=d$ and $w_i=v_i$ for every $i$.

\eqref{SignoInduccion}
By \eqref{CumulusDescomposicion}, $w_1$ and $w'$ have to be the elements in the proof of \eqref{CumulusDescomposicion}.
Using this it easily follows that if $E(w_1)=X$ and $B(w')=Y^{-1}$ then
$N(w)=N(w_1)+N(w')+1$. Therefore $\sgn(w)=-\sgn(w_1)\sgn(w')$.
A similar argument shows that in all the other cases $N(w)\equiv N(w_1)+N(w')\mod 2$ and therefore $\sgn(w)=\sgn(w_1)\sgn(w')$.
\end{proof}

Let now $f=\sum_{w\in F} f_{w} w\in K\GEN{X^{\pm 1},Y^{\pm 1}}\setminus \{0\}$, with each $f_{w}\in K$ and define
	$$C(f)=\max \{C(w): w\in \Supp(f)\}.$$
For every $b,e\in \{X,X^{-1},Y,Y^{-1}\}$ we denote
$$f_{b,e} = \sum_{\matriz{{c}C(w)=C(f)\\B(w)=b, E(w)=e}} \sgn(w) f_w$$
Finally, we define the following four elements of $K$:
\begin{eqnarray*}
f_1&=&f_{X,X}+f_{X,Y}+f_{Y,X}+f_{Y,Y}-f_{Y\inv,X}-f_{Y\inv,Y}, \\
f_2&=&f_{X,X\inv}+f_{X,Y}+f_{X,Y\inv}+f_{Y,X\inv}+f_{Y,Y}+f_{Y,Y\inv}-f_{Y\inv,X\inv}-f_{Y\inv,Y}-f_{Y\inv,Y\inv}, \\
f_3&=&f_{X\inv,X}+f_{X\inv,Y}+f_{Y\inv,X}+f_{Y\inv,Y}, \\
f_4&=&f_{X\inv,X\inv}+f_{X\inv,Y}+f_{X\inv,Y\inv}+f_{Y\inv,X\inv}+f_{Y\inv,Y}+f_{Y\inv,Y\inv}.
\end{eqnarray*}

We are ready to present the main result of this section.

\begin{proposition}\label{ThExamples}
If $f\in K\GEN{X^{\pm 1},Y^{\pm 1}}$ is a LPI of $\U(\FC)$ then $f_i=0$ for every $i \in \{1,2,3,4\}$.
\end{proposition}

\begin{proof}
Let $\varphi:\FC\rightarrow M_2(K[T])$ be the algebra homomorphism defined in the proof of Lemma~\ref{thekey} and let
	\begin{equation}\label{uv}
	u=(1+\alpha\beta\alpha)(1+\beta) \qand v=(1+\alpha\beta\alpha)(1+(1-\alpha)\beta(1+\alpha))
	\end{equation}
Let $\Phi$ denote the unique algebra homomorphism $\Phi:KF\rightarrow M_2(K(T))$
with  $\Phi(X)=\varphi(u)$  and $\Phi(Y) = \varphi(v)$.
We identify $M_2(K[T])$ with the polynomial ring $M_2(K)[T]$.
So we can talk of the leading term of a non-zero element of $M_2(K[T])$.

Let $e_{ij}$ denote the matrix having $1$ at the $(i,j)$ entry and zeros elsewhere.
Hence
\begin{eqnarray*}
\Phi(X)&=&1+(e_{12}+e_{21})T+e_{11}T^2 \\
\Phi(Y)&=&1+(e_{21}+e_{22}-e_{11})T+(e_{11}+e_{12})T^2.
\end{eqnarray*}

\textbf{Claim}: If $w$ is an element of $F$ of cumulus $c$ then the leading term in $\Phi(w)$ is as displayed in Table~\ref{Principal}.
\begin{table}[h!]
$$\matriz{{|c|c|l|}
	\hline
	B(w) & E(w) & \text{Leading term of } \Phi(w) \\\hline\hline
	& X                                	& T^{2c} \; \sgn(w) \;e_{11} \\
	X \text{ or } Y	& X\inv \text{ or } Y\inv  	&  T^{2c} \; \sgn(w) \;e_{12}\\
	& Y                                	& T^{2c} \; \sgn(w) \;(e_{11}+e_{12}) \\\hline
	& X                                	& T^{2c} \; \sgn(w) \;e_{21} \\
	X\inv              & X\inv \text{ or } Y\inv  	& T^{2c} \; \sgn(w) \;e_{22} \\
	& Y                                	& T^{2c} \; \sgn(w) \;(e_{21}+e_{22})  \\\hline
	& X                                	& T^{2c} \; \sgn(w) \;(e_{21}-e_{11})  \\
	Y\inv              & X\inv \text{ or }Y\inv   	& T^{2c} \; \sgn(w) \;(e_{22}-e_{12})  \\
	& Y                                	& T^{2c} \; \sgn(w)  (e_{21}+e_{22}-e_{11}-e_{12})  \\\hline
}$$
\caption{\label{Principal}}
\end{table}

Using the claim it easily follows that if $f\in KF$ with $C(f)=c$ then
$\Phi(f)$ is a polynomial of degree at most $2c$ with coefficients in $M_2(K)$ and the coefficient of $T^{2c}$ in these polynomial is the matrix
$$\pmatriz{{cc} f_1 & f_2 \\ f_3 & f_4}.$$
As $u,v\in \U(\FC)$, if $f$ is a LPI of $\U(\FC)$ then we have $\Phi(f)=\varphi(f(u,v))=0$ and therefore, $f_i=0$ for every $i=1,2,3,4$, as desired.

So we only have to prove the claim and for that we argue by induction on $c=C(w)$.
The case where $c=1$ can be checked directly.
Assume that $c>1$ and write $w=w_1w'$ with $C(w_1)=1$ and $C(w')=c-1$.
Then $B(w)=B(w_1)$, $E(w)=E(w')$.
Moreover, by induction hypothesis, the leading terms of $\Phi(w_1)$ and $\Phi(w')$ agrees with Table~\ref{Principal}.
We will show that the leading term of $\Phi(w)$ also agrees with Table~\ref{Principal} by calculating the product of the leading terms of $\Phi(w_1)$ and $\Phi(w')$. It turns out that this product is non-zero and hence it is the leading term of $\Phi(w)$ and the calculation will show that it agrees with Table~\ref{Principal}.

There are six possibilities for $w_1$ and nine cases for the leading term of $\Phi(w')$ according to the table.
However some of the possibilities are not compatible with the condition $C(w_1w')=c$.
We check the claim in all the cases considering separately the different values for $B(w')$.

Assume that $B(w')$ is either $X$ or $Y$.
By induction hypothesis, the leading term of $\Phi(w')$ is $\sgn(w')e_{11}T^{2(c-1)}$, if $E(w')=X$;
$\sgn(w')e_{12}T^{2(c-1)}$, if $E(w')$ is either $X\inv$ or $Y\inv$;
and $\sgn(w')(e_{11}+e_{12})T^{2(c-1)}$, if $E(w')=Y$.
Moreover, $w_1\in \{X,Y,X\inv Y, Y\inv X\}$.
This gives rise to  twelve cases depending on the end of $w'$ and the value of $w_1$.
We have to multiply the leading terms of $\Phi(w_1)$ and $\Phi(w')$ and
check that the resulting product agrees with the value in the previous table.
Table \ref{Tabla1} summarizes these products.
The second row represents the leading terms of $\Phi(w')$ according to the induction hypothesis and
the second column display the leading terms of $\Phi(w_1)$.
The resulting product is displayed in the standard way. %and it is obtained using (\ref{SignoInduccion}):
\begin{table}[h!]
	$$\matriz{{|c|l||l|l|l|}
		\hline
		& E(w') & X & X\inv \text{ or } Y\inv & Y \\\hline
		w_1 & & T^{2c-2}\; \sgn(w')\; e_{11} & T^{2c-2}\; \sgn(w')\; e_{12} & T^{2c-2}\; \sgn(w') \\
		& &                 &                 & (e_{11}+e_{12}) \\\hline \hline
		X & T^{2}\; \sgn(w_1)\; e_{11}  &  T^{2c}\; \sgn(w)\;  e_{11} & T^{2c}\; \sgn(w)\;  e_{12} & T^{2c}\; \sgn(w)\\
		& &                                &                                 &  (e_{11}+e_{12})\\\hline
		Y & T^{2}\; \sgn(w_1) & T^{2c}\; \sgn(w)\; e_{11} & T^{2c}\; \sgn(w)\;  e_{12} & T^{2c}\; \sgn(w)\\
		& (e_{11}+e_{12})&                                &                                 &  (e_{11}+e_{12}) \\\hline
		X\inv Y & T^{2}\; \sgn(w_1) &  T^{2c}\; \sgn(w)\; e_{21} & T^{2c}\; \sgn(w)\; e_{22} &T^{2c}\; \sgn(w)\\
		& (e_{21}+e_{22})&                                &                                 &  (e_{21}+e_{22}) \\\hline
		Y\inv X & T^{2}\; \sgn(w_1)& T^{2c}\; \sgn(w) & T^{2c}\;\sgn(w) & T^{2c}\; \sgn(w) \\
		& (e_{21}-e_{11})& (e_{21}-e_{11}) &  (e_{22}-e_{12}) &(e_{21}+e_{22}-e_{11}-e_{12})
		\\\hline
	}$$
	\caption{\label{Tabla1} $B(w')\in \{X,Y\}$}
\end{table}

Assume now that $B(w')=X\inv$.
In this case the induction hypothesis is that the leading term of $\Phi(w')$
is
%\ChAngel{\sout{$\sgn(w')e_{21}T^{2(c-1)}$ if $E(w')=X$;
%$\sgn(w')e_{22}T^{2(c-1)}$ if $E(w')$ is either $X\inv$ or $Y\inv$ and $\sgn(w')(e_{21}+e_{22})T^{2(c-1)}$ if $E(w')=Y$.
%On the other hand}
as in the second row of Table \ref{Tabla2}. Moreover, $w_1\in \{X\inv,Y,Y\inv,X\inv Y\}$.
Now Table \ref{Tabla2} consider the different options.
\begin{table}[h!]
	$$\matriz{{|c|l||l|l|l|}\hline
		& E(w') & X & X\inv \text{ or } Y\inv & Y \\\hline
		w_1&& T^{2c-2} \; \sgn(w') \; e_{21} & T^{2c-2} \; \sgn(w') \; e_{22} & T^{2c-2} \; \sgn(w') \\
		&&                               &                               & (e_{21}+e_{22}) \\\hline\hline
		X\inv& T^2 \; \sgn(w_1) e_{22} & T^{2c} \; \sgn(w)  e_{21} & T^{2c} \; \sgn(w) e_{22} & T^{2c} \; \sgn(w) \\
		&&                               &                               & (e_{21}+e_{22})\\\hline
		Y  & \sgn(w_1)  & T^{2c} \; \sgn(w) e_{11} & T^{2c} \; \sgn(w) e_{12} & T^{2c} \; \sgn(w) \\
		&(e_{11}+e_{12})&                               &                               & (e_{11}+e_{12}) \\\hline
		Y\inv & T^2 \; \sgn(w_1) & T^{2c} \; \sgn(w)  & T^{2c} \; \sgn(w)  & T^{2c} \; \sgn(w)\\
		& (e_{22}-e_{12})    & (e_{21} -e_{11})     &(e_{22} -e_{12})   &(e_{21}+e_{22} -e_{11}-e_{12})\\\hline
		X\inv Y & T^2 \; \sgn(w_1)  & T^{2c} \; \sgn(w)  e_{21} & T^{2c} \; \sgn(w) e_{22} & T^{2c} \; \sgn(w)\\
		& (e_{21}+ e_{22})    &                             &                            & (e_{21}+e_{22}) \\\hline
	}$$
	\caption{\label{Tabla2} $B(w')=X^{-1}$.}
\end{table}

Finally, suppose that $B(w')=Y\inv$.
In this case Table \ref{Tabla3} displays the different options for $E(w')$ and $w_1$.
In all the previous cases we have $\sgn(w)=\sgn(w_1)\sgn(w')$, by (\ref{SignoInduccion}). However, in the present case this only holds for $w_1\in
\{X\inv,Y\inv\}$.
Otherwise $\sgn(w)=-\sgn(w_1)\sgn(w')$.
\begin{table}[h!]
	$$\matriz{{|c|l||l|l|l|}\hline
		& E(w') & X & X\inv \text{ or } Y\inv & Y \\\hline
		w_1 & & T^{2c-2} \; \sgn(w') & T^{2c-2} \; \sgn(w') & T^{2c-2} \; \sgn(w')\\
		& &(e_{21} -e_{11}) &(e_{22} -e_{12}) &(e_{21}+e_{22} -e_{11}-e_{12}) \\ \hline \hline
		X & T^2 \; \sgn(w_1)e_{11} & T^{2c} \; \sgn(w)  e_{11} & T^{2c} \; \sgn(w) e_{12} & T^{2c} \; \sgn(w) (e_{11}+e_{12})\\\hline
		X\inv & T^2 \; \sgn(w_1)e_{22} & T^{2c} \; \sgn(w) e_{21} & T^{2c} \; \sgn(w) e_{22} & T^{2c} \; \sgn(w) (e_{21}+e_{22})  \\\hline
		Y\inv & T^2 \; \sgn(w_1) & T^{2c} \; \sgn(w)  & T^{2c} \; \sgn(w)  & T^{2c} \; \sgn(w) \\
		&(e_{22} -e_{12}) & (e_{21} -e_{11}) & (e_{22} -e_{12})  & (e_{21}+e_{22} -e_{11}-e_{12}) \\\hline
		Y\inv X & T^2 \; \sgn(w_1) & T^{2c} \; \sgn(w)  & T^{2c} \; \sgn(w)  & T^{2c} \; \sgn(w) \\
		&(e_{21}-e_{11}) & (e_{21}-e_{11}) & (e_{22} -e_{12})  & (e_{21}+e_{22} -e_{11}-e_{12}) \\\hline
	}$$
	\caption{\label{Tabla3} $B(w')=Y^{-1}$.}
\end{table}
This finishes the proof of the claim.
\end{proof}

As a consequence of Proposition~\ref{ThExamples} we obtain at once the following corollary, which in turns will be the main ingredient for the proof of Proposition~\ref{Support4}.

\begin{corollary}\label{onemaxcum}
	Let $f\in K\GEN{X^{\pm 1},Y^{\pm 1}}$ and suppose that the support of $f$ contains only one element $w$ with $C(w)=C(f)$.
	Then $f$ is not a LPI of $\U(\FC)$.
\end{corollary}

We finish with the

\begin{proofof}\textbf{Proposition~\ref{Support4}}.
We argue by contradiction.
So let $f$ be a non-zero LPI of $\U(\FC)$ with less than four elements in the support.
As $\U(\FC)$ contains a non-abelian free group, it does not satisfy a group identity and hence $\U(\FC)$ does not satisfy a LPI with support of cardinality smaller than three.
Thus the support of $f$ has exactly three elements.
One may assume without loss of generality that $1$ belongs to the support of $f$.
As any free group is contained in a free group of rank $2$ we may assume also that $f\in K\GEN{X^{\pm 1},Y^{\pm 1}}$.
By Corollary~\ref{onemaxcum}, the support of $f$ has two different elements with the same cumulus as $f$.
Thus, one may assume that $f=1+\alpha_1w_1+\alpha_1w_2$ with $w_1$ and $w_2$ different non-trivial monomials with the same cumulus and $\alpha_1,\alpha_2\in K\setminus \{0\}$.
We may assume that the common cumulus, say $c$, of $w_1$ and $w_2$ is minimal.
Then $w_1^{-1}f=\alpha_1+w_1^{-1}+\alpha_2w_1^{-1}w_2$ and $fw_1^{-1}=\alpha_1+w_1^{-1}+\alpha_2w_2w_1^{-1}$ are LPI identities of $\U(\FC)$ too.
By Lemma~\ref{CumulusSignPropiedades}.\eqref{CumulusInverse} and Corollary~\ref{onemaxcum}, we have
	\begin{equation}\label{CumulusIguales}
	c=C(w_1)=C(w_2)=C(w_1^{-1}w_2)=C(w_2w_1^{-1}).	
\end{equation}
This implies that $c>1$.
Using the unique decomposition of $w_1$ (respectively, $w_2$) in product of $C(w_1)$ (respectively, $C(w_2)$) elements of cumulus 1 we obtain unique factorizations
$w_1=w_ba_1\cdots a_kw_e$ and $w_2=w_bb_1\dots b_k w_e$ where $a_1\ne b_1$, $a_k\ne b_k$, each $a_i$ and $b_i$ has cumulus 1, and $k=c-C(w_b)-C(w_e)>0$.

We claim that $c$ is either $2(C(w_e)+k)$ or $2(C(w_e)+k)-1$.
Suppose that $c \ne 2(C(w_e)+k)$.
As $c=C(w_1^{-1}w_2)$, we have $B(a_1)=B(b_1)$ and as $a_1\ne b_1$ this implies that $\{a_1,b_1\}$ is either $\{X^{-1},X^{-1}Y\}$ or $\{Y^{-1},Y^{-1}X\}$.
By symmetry one may assume that $a_1=X^{-1}$ and $b_1=X^{-1}Y$.
If $a_2\cdots a_kw_e=1$ then $w_1^{-1}w_2=Y$ and hence $c=C(w_1^{-1}w_2)=1$, a contradiction.
Therefore, $a_2\cdots a_kw_e\ne 1$, $E(w_e^{-1}a_k^{-1}  \cdots a_2^{-1})\not\in \{X^{-1},Y^{-1}\}$ and $B(b_2\cdots b_kw_e)\ne Y^{-1}$.
Hence
	$$c=C(w_1^{-1}w_2)=C(w_e^{-1}a_k^{-1}  \cdots a_2^{-1}Yb_{2}\cdots b_k w_e)=2(C(w_e)+k)-1,$$
as desired.

As $c=C(w_2w_1^{-1})$, the same proof of the previous paragraph shows that $c$ is either $2(C(w_b)+k)$ or $2(C(w_b)+k)-1$. Thus $C(w_e)=C(w_b)$ and  $2C(w_e)+k=c\in \{2(C(w_e)+k),2(C(w_e)+k)-1\}$. As $k\ne 0$ we deduce that $k=1$ and  $2C(w_e)+1=c=C(w_1^{-1}w_2)=C(w_e^{-1}a_1^{-1}b_1w_e)$.
Therefore $C(a_1^{-1}b_1)=1$.
Moreover $\alpha_1^{-1}a_1^{-1}w_b^{-1}fw_e^{-1}= 1+\alpha_1^{-1}\alpha_2a_1^{-1}b_1+\alpha_1^{-1}a_1^{-1}w_b^{-1}w_e^{-1}$ is another LPI of $\FC$ with three elements in the support.
By Corollary~\ref{onemaxcum}, $C(a_1^{-1}w_b^{-1}w_e^{-1})=C(a_1^{-1}b_1)=1<c$.
This contradicts the minimality of $c$.
%Thus $\{w_1,w_2\}$ is either $\{X^{-1},X^{-1}Y\}$ or $\{Y^{-1},Y^{-1}X\}$.
%By symmetry one may assume that $w_1=X^{-1}$ and $w_2=X^{-1}Y$.
%Then $C(w_2w_1^{-1})=C(X^{-1}YX)=2\ne c$, contradicting \eqref{CumulusIguales}.
\end{proofof}

\begin{remark}
Proposition~\ref{ThExamples} provides some constrains for the LPIs of $\U(\FC)$.
There are several obvious ways to obtain other constrains.
For example, if  $f=f(X,Y)$ is a LPI of $\FC$ then so are $g=f(Y,X)$ and $h=f(X^{-1},Y)$. Applying Proposition~\ref{ThExamples} to these Laurent polynomials we obtain that the expressions obtained by interchanging the roles of $X$ and $Y$ or interchanging the roles of $X$ and $X^{-1}$  in each $f_i$ should vanish.
Furthermore, the proof of Proposition~\ref{ThExamples} can be easily adapted to obtain more constrains for the LPIs of $\U(\FC)$.
More precisely, take any two units $u$ and $v$ of $\FC$ and let $\Phi$ be the homomorphism $\Phi:KF\rightarrow M_2(K[T])$ given by $\Phi(X)=\varphi(u)$ and $\Phi(Y)=\varphi(v)$.
An appropriate election of $u$ and $v$ implies a bound on the degree of $\Phi(f)$ in terms of some number depending on $f$.
For example, if $u$ and $v$ are as in \eqref{uv} then the degree of $\Phi(f)$ is at most twice the cumulus of $f$.
Another possible election for $u$ and $v$ is as follows:
	$$u=(1+\beta)(1+\alpha\beta\alpha)\qand v=(1+(1-\alpha)\beta(1+\alpha))(1+(1+\alpha)\beta(1-\alpha)).$$
In this case, if we define
$C'(w)=\sum_{i=0}^k |n_i| + \sum_{i=1}^k |m_i|$, for $w$ as in \eqref{Word}, and $C'(f)=\max\{C'(w) : w\in \Supp(f)\}$, then the degree of $\Phi(f)$ is at most $2C'(f)$.
Calculating the coefficient of $\Phi(f)$ of degree $2C'(f)$ we obtain a matrix which should vanish if $f$ is a LPI of $\U(\FC)$.
\end{remark}

\bibliographystyle{amsalpha}
\bibliography{ReferencesMSC}
\end{document}